\documentclass[12pt]{article}
\usepackage{amssymb,latexsym,hyperref}

\def\header{Universal Functions \hfill}
\def\al{\alpha}
\def\be{\beta}
\def\cc{{\mathfrak c}}
\def\comb(#1,#2){ \left( \begin{array}{c} #1 \\ #2  \end{array} \right) }
\def\ctbl(#1){{{Ctbl}(#1)}}
\def\cube{{\mathcal C}}
\def\de{\delta}
\def\dom(#1){dom(#1)}
\def\fin(#1){{{Fin}(#1)}}
\def\ga{\gamma}

\def\ka{\kappa}
\def\la{\langle}
\def\om{\omega}
\def\pair(#1,#2){\la #1, #2 \ra}
\def\picube(#1){{\bf \Pi}^0_{#1}(\cube)}
\def\pirec(#1){{\bf \Pi}^0_{#1}(\rec)}
\def\Poset{{\mathbb P}}
\def\poset{{\mathbb P}}
\def\pow(#1){{\mathcal P}(#1)}
\def\pp{\mathfrak p}
\def\proof{\par\noindent Proof\par\noindent}
\def\pr{\prime}
\def\qed{\par\noindent QED\par\bigskip}
\def\Rationals{{\mathbb Q}}
\def\ra{\rangle}
\def\rcantorpar{{(2^\om)}}
\def\rcantor{{2^\om}}
\def\Reals{{2^\om}}
\def\rec{{\mathcal R}}
\def\res{\mathord{\upharpoonright}}
\def\rmand{\mbox{ and }}
\def\rmforall{\mbox{ for all }}
\def\rmiff{\mbox{ iff }}
\def\rmif{\mbox{ if }}
\def\rr{{\mathbb R}}
\def\SetOf#1#2{\left\{#1 \ \left| \ #2 \right.\right\}}
\def\sicube(#1){{\bf \Pi}^0_{#1}(\cube)}
\def\sirec(#1){{\bf \Sigma}^0_{#1}(\rec)}
\def\si{\sigma}
\def\sm{\setminus}
\def\st{\;:\;} 
\def\su{\subseteq}
\def\text{\rm}

\newtheorem{theorem}{Theorem}[section]

\newtheorem{theor}[theorem]{Theorem}
\newtheorem{lemma}[theorem]{Lemma}
\newtheorem{define}[theorem]{Definition}
\newtheorem{remark}[theorem]{Remark}
\newtheorem{prop}[theorem]{Proposition}
\newtheorem{ques}[theorem]{Question}

\newtheorem{cor}[theorem]{Corollary}

\def\addresspaul{\begin{flushleft}
Paul B. Larson\\
larsonpb@muohio.edu\\
http://www.users.muohio.edu/larsonpb\\
Department of Mathematics\\
Miami University\\
Oxford, Ohio 45056\\
\end{flushleft}}

\def\addressarn{\begin{flushleft}
Arnold W. Miller \\
miller@math.wisc.edu \\
http://www.math.wisc.edu/$\sim$miller\\
University of Wisconsin-Madison \\
Department of Mathematics, Van Vleck Hall \\
480 Lincoln Drive \\
Madison, Wisconsin 53706-1388 \\
\end{flushleft}}

\def\addressjuris{\begin{flushleft}
Juris Stepr{\={a}}ns \\
steprans@yorku.ca \\
York University\\
Department of Mathematics\\
4700 Keele Street\\
Toronto, Ontario \\
Canada M3J 1P3\\
\end{flushleft}}

\def\addressbill{\begin{flushleft}
William A.R. Weiss\\
weiss@math.toronto.edu\\
Department of Mathematics\\
University of Toronto, Ontario\\
Canada M5S 3G3\\
\end{flushleft}}

\begin{document}

\begin{center}
{\large Universal Functions}
\end{center}

\begin{flushright}
Paul B. Larson\\
Arnold W. Miller\\
Juris Stepr\={a}ns\\
William A.R. Weiss
\end{flushright}

{\small \tableofcontents}

\begin{center}  Abstract\footnote{
\par Results obtained Mar-Jun 2009, Nov 2010.  Last revised April 2012.}
\end{center}

A function of two variables $F(x,y)$ is universal
iff for every other function $G(x,y)$ there exists functions $h(x)$ and
$k(y)$ with $$G(x,y)=F(h(x),k(y)).$$
Sierpinski showed that assuming the continuum hypothesis there exists
a Borel function $F(x,y)$ which is universal.
Assuming Martin's Axiom
there is a universal function of Baire class 2. A
universal function cannot be of Baire class 1.
Here we show that it is consistent that for each $\alpha$ with
$2<\alpha<\omega_1$  there is a universal function of class $\alpha$
but none of class $\beta<\alpha$. We show that it is consistent with
ZFC that there is no universal function (Borel or not) on the reals, 
and we show that it is consistent that there is a universal function but no
Borel universal function.  We also prove some results concerning
higher arity universal functions.  For example, the existence of
an $F$ such that for every $G$ there are $h_1,h_2,h_3$ 
such that for all $x,y,z$
$$G(x,y,z)=F(h_1(x),h_2(y),h_3(z))$$
is equivalent to the existence of a 2-ary universal $F$, however
the existence of
an $F$ such that for every $G$ there are $h_1,h_2,h_3$ 
such that for all $x,y,z$
$$G(x,y,z)=F(h_1(x,y),h_2(x,z),h_3(y,z))$$ 
follows from a 2-ary universal $F$ but is strictly weaker.

\section{Introduction}

\begin{define}\label{defunivfcn}
A function $F:X \times X\to X$ is universal iff
for any $$G:X \times X\to X$$ there is
$g:X\to X$ such that for all  $x,y\in X$
 $$G(x,y)=F(g(x),g(y)).$$
\end{define}

Sierpinski asked\footnote{Scottish book, Mauldin \cite{mauldin}
problem 132.} when $X$ is the real line if there always is a
Borel function which is universal.
He had shown that there is a Borel universal function assuming the
continuum hypothesis (Sierpinski \cite{sier}).

Without loss we may use different functions
on the $x$ and $y$ coordinates, i.e., $G(x,y)=F(g_0(x),g_1(y))$ in
the definition of universal function $F$.
To see this suppose we are given $F^*$ such that
for every $G$ we may find $g_0,g_1$ with $G(x,y)=F^*(g_0(x),g_1(y))$
for all $x,y$.  Then we can construct a universal $F$ which uses
only a single $g$.
Take a bijection, i.e., pairing function between
$X\times X$ and $X$, i.e., $(x_0,x_1)\mapsto\pair(x_0,x_1)$.
Define
$$F(\pair(x_0,x_1),\pair(y_0,y_1))=F^*(x_0,y_1)$$
where $\pair(x_0,x_1)$ is a pairing function. Given any
$g_0,g_1$ define $$g(u)=\la g_0(u),g_1(u) \ra$$
and note that
$$F(g(x),g(y))=F^*(g_0(x),g_1(y))$$
for every $x,y$.

In the case $X=2^\om$ there is a pairing function 
which is a homeomorphism and hence the Borel complexity
of $F$ and $F^*$ are the same.
For abstract universal $F$
a pairing function exists for any infinite $X$ by the
axiom of choice.

In section \ref{boreluniv} we show that the existence
of a Borel universal functions is equivalent to under a
weak cardinality assumption to the statement that
every subset of the plane is in the $\si$-algebra generated
by the abstract rectangles.  We also show that a universal
function cannot be of Baire class 1.

In section \ref{MA} we prove some results concerning Martin's
axiom and universal function.  We show that
although MA implies that there is a universal function of
Baire class 2 it is consistent to have ${\text MA}_{\aleph_1}$
hold but no Borel universal functions.

In section \ref{special} we consider universal functions of
a special kind.  For example, $F(x,y)=k(x+y)$.
We also discuss special versions due to Todorcevic and
Davies.

In section \ref{abstract} we consider abstract universal functions,
i.e., those defined on a cardinal $\ka$ with no notion of definability, 
Borel or otherwise.   We show that if $2^{<\ka}=\ka$, then they
exists.  We also show that it is consistent that none exists for
$\ka$ equal to the continuum. We also prove some
weak abstract versions of universal functions from the 
assumption ${\text MA}_{\aleph_1}$. 

In section \ref{higher} we take up the problem of universal
functions of higher arity.  We show that there is a natural hierarchy
of such notions and we show that this hierarchy is strictly descending.

\section{Borel Universal Functions}\label{boreluniv}

\begin{define} \label{defrec}
We let $\rec$ denote the family of abstract rectangles,
$$\rec=\{A\times B\st  A,B\su 2^\om\}.$$
\end{define}

\begin{define}
  $\sirec(\al)$ and $\pirec(\al)$ for $\al<\om_1$ are inductively defined by:
\begin{itemize}
\item $\sirec(0)=\pirec(0)=$ the finite boolean combinations
of sets from $\rec$,
\item $\sirec(\al)$ is the countable unions of
sets from $\pirec(<\al)=\bigcup_{\be<\al}\pirec(\be)$, and
\item $\pirec(\al)$ is the countable intersections of
sets from $\sirec(<\al)$.
\end{itemize}
\end{define}

\begin{define}
A Borel function $F:2^\om \times 2^\om\to 2^\om$ is at the
$\al$-level iff
for any $n$ the set $\{(u,v)\st F(u,v)(n)=1\}$ is
${\bf\Sigma}^0_\al$.
\end{define}
We remark that a Borel function at level $\al$ is in
Baire class $\al$, but not the converse.  In the context
of $2^\om$ a function is of Baire class $\al$ iff the
preimage of every clopen set is ${\bf\Delta}_{\al+1}$.
For more on the classical theory
of Baire class $\al$, see Kechris \cite{kechris} p. 190.

\begin{theorem}\label{mainthm}
Suppose that $2^{<\cc}=\cc$, then the following are equivalent:
\begin{enumerate}
\item There is a  Borel function $F:2^\om \times 2^\om\to 2^\om$
which is universal.
\item Every subset of the plane $2^\om \times 2^\om$ is
in the $\si$-algebra generated by the abstract rectangles, $\rec$.
\end{enumerate}
Furthermore, $\pow(2^\om\times 2^\om)=\sirec(\al)$
iff $F$ can be taken to be the $\al$-level.
\end{theorem}

\proof

$(1)\to (2)$.

\noindent
Suppose there is a Borel universal
$F:2^\om \times 2^\om\to 2$.
Let $A\su 2^\om\times 2^\om$ be arbitrary and suppose
$g:2^\om\to 2^\om$ has the property that
$$\forall x,y\;\;\;\;(x,y)\in A \rmiff F(g(x),g(y))=1.$$
Let $B$ be the Borel set $F^{-1}(1)$. Then $B$ is generated by countable unions
and intersections from sets of the form $C \times D$, for $C$, $D$ clopen subsets 
$2^{\om}$. Note
that $(x,y)\in A$ iff $(g(x),g(y))\in B$.
Define $h(x,y)=(g(x),g(y))$ and note that
$$h^{-1}(C\times D)=g^{-1}(C)\times g^{-1}(D)$$
for all sets $C, D \subseteq 2^{\om}$.
Since
$$A=h^{-1}(B),$$
and since preimages pass over countable unions and intersections
it follows that $A$ is in the $\si$-algebra of abstract
rectangles. Furthermore if $B$ is ${\bf \Sigma}^0_\al$, then
$A$ is $\sirec(\al)$.

\bigskip

$(2)\to (1)$.

\noindent
We show first that there exists an $X\su 2^\om$ of cardinality
$\cc$ which has the property that every $Y\su X$ of cardinality
strictly smaller than $\cc$ is Borel relative to $X$,
i.e., is the intersection of a
Borel set with $X$. See Bing, Bledsoe, and Mauldin \cite{bbm}.
Let $A\su\cc\times\cc$ be such that for every $B\in[\cc]^{<\cc}$
there exists $\al<\cc$ such that
$$B=A_{\al}=^{def}\{\be\st (\al,\be)\in A\}.$$
This is possible because $2^{<\cc}=\cc$.  Since $A$ is in
the $\si$-algebra generated by the abstract rectangles, there
exists a sequence $A_n\su \cc$ for $n<\om$
such that $A$ is in the $\si$-algebra generated by
$\{A_n\times A_m\st n,m<\om\}$.  Let
$f:\cc\to 2^\om$ be the Marczewski characteristic function for
the sequence $(A_n:n<\om)$, i.e.,
$$f(x)(n)=
\left\{
\begin{array}{ll}
0 & \mbox{ if } x\notin A_n \\
1 & \mbox{ if } x\in A_n
\end{array}\right.$$
Let $X=f(\cc)$. Let us check that $X$ has the required property.
Let $Y$ be a subset of $X$ of cardinality less than $\cc$, and let 
$B$ be a subset of $\cc$ of cardinality less than $\cc$ such that 
$Y = f(B)$. Each set of the form $A_{n} \times A_{m}$ is the preimage under
$f$ of a clopen subset of $2^{\om} \times 2^{\om}$. 
Again using the fact that preimages pass over countable unions and intersections, 
we can find a Borel subset $2^{\om} \times 2^{\om}$ whose preimage under $f$ is $A$. 
Then $Y$ will be one section of this set, intersected with $X$.  
Also note that if $A$ is $\sirec(\al)$, then every subset of
$X$ of cardinality strictly smaller than $\cc$ is ${\bf\Sigma}_\al^0$
relative to $X$.

Now let $U\subseteq 2^\om\times 2^\om$ be a universal ${\bf\Sigma}_\al^0$
set.  Define $G:2^\om\times 2^\om\to 2^\om$ by
$$\forall n \;\;\;\;(G(x,y)(n)=1 \rmiff (x_n,y)\in U)$$
where $x\mapsto (x_n:n<\om)\in (2^\om)^\om$ is a homeomorphism.

Let $f_1:\cc^2\to 2^\om$ be an arbitrary function with the property
that $\al>\be\to f_1(\al,\be)=\vec{0}$ (the identically zero map).
We claim that there exists $h_1,h_2:\cc\to 2^\om$ such that
$$f_1(\al,\be)=G(h_1(\be),h_2(\al)) \rmforall (\al,\be)\in \cc^2.$$
To see this, let $X=\{x_\ga:\ga <\cc\}$.  Let $h_2(\al)=x_\al$.
For each $n$ and $\be$ note that
$$B_n=^{def}\{x_\al : f_1(\al,\be)(n)=1\}$$
is a subset of $X$ of cardinality less that $\cc$ and so
there exists $y_n\in 2^\om$ such that
$B_n=X\cap U_{y_n}$.  Construct $h_1(\be)=y$ corresponding to
such a sequence $(y_n:n<\om)$.

By an analogous argument, if $f_2:\cc^2 \to 2^\om$ is an arbitrary
map with the property that $\be>\al\to f_2(\al,\be)=\vec{0}$,
then there exists $k_1,k_2:\cc\to 2^\om$
such that
$$f_2(\al,\be)=G(k_1(\al),k_2(\be))\rmforall (\al,\be)\in Q_2 .$$

Now define $F:2^\om\times 2^\om\to 2^\om$ by:
$$F((x_1,y_1),(x_2,y_2))=\max(G(x_2,x_1),G(y_1,y_2))$$
where $\max:2^\om\times 2^\om \to 2^\om$ is the pointwise
maximum, i.e., $\max(u,v)=w$ iff $w(n)$ is the maximum of
$u(n)$ and $v(n)$ for each $n<\om$. Then

$\;\;\;\;\; F((x_1,y_1),(x_2,y_2))(n)=1$ iff
$G(x_1,x_2)(n)=1$ or $G(y_2,y_1)(n)=1$.

We show the $F$ is universal.
Given an arbitrary $f:\cc\times\cc\to 2^\om$ we can
find $f_1$ and $f_2$ as above so that
$$f(\al,\be)=\max(f_1(\al,\be),f_2(\al,\be)) \rmforall (\al,\be)\in \cc^2.$$
Define
$l_1(\al)=(h_2(\al),k_1(\al))$ and $l_2(\be)=(h_1(\be),k_2(\be))$.
Then
$$f(\al,\be)=F(l_1(\al),l_2(\be)) \rmforall \al,\be<\cc.$$

Also $F$ is at the $\al$-level, i.e.,
for any $n$ the set $\{(u,v)\st F(u,v)(n)=1\}$ is
${\bf\Sigma}^0_\al$.

\qed

\begin{cor}
It is consistent that for each $\al$ with $2<\al<\om_1$ there is a universal
function of Baire class $\al$ but none of class $\be<\al$.
It is consistent that there is a universal function but no Borel
universal function.  If $\pp = \cc$,
then there is a universal function of Baire class 2.
\end{cor}
\proof
This follows from corresponding results about
the $\si$-algebra of abstract rectangles, see Miller \cite{miller}
Theorem 37.
The existence of an abstract universal function follows
from $\cc^{<\cc}=\cc$ Theorem \ref{existuniv} and this holds
in many models in which not every subset of the plane is in the
$\sigma$-algebra generated by the abstract rectangles.  For example,
Kunen in his thesis showed this is true in the Cohen real model.
The cardinal $\pp$ is the psuedo-intersection number.  An
equivalent definition for it is
the smallest cardinal for which Martin's Axiom of for $\si$-centered
posets fails.  This is due to Bell \cite{bell}),
for the proof Bell's Theorem see also Weiss \cite{weiss}.
Proposition \ref{proprao} shows that if $\pp = \cc$,
then there is a universal function of Baire class 2.
\qed

\begin{ques}
Suppose there is a universal function of Baire class $\al$. Then
is there a universal function of level $\al$?
\end{ques}

\begin{prop}
A universal function cannot be of Baire class 1.
\end{prop}
\proof
Suppose that $F$ is of Baire class 1. Let $\{h_\xi\}_{\xi\in\mathfrak c}$
enumerate all functions from a countable subset of $\rcantor$ whose range is dense
in itself.
Let $\{r_\xi\}_{\xi\in\mathfrak c}$ enumerate all $\rcantor$.
For each $\xi$ partition the domain of $h_\xi$ into $A_\xi$ and $B_\xi$
 such that
 $$\overline{\SetOf{h_\xi(x)}{x\in A_\xi}} =
 \overline{\SetOf{h_\xi(x)}{x\in B_\xi}}$$
and let $G:\rcantor^2 \to \rcantor$ be any function satisfying
$G(r_\xi,r) = 1$ if $r \in A_\xi$ and $G(r_\xi,r) = 0$ if $r \in B_\xi$.

Now suppose that $h:\rcantor\to \rcantor$ witnesses that $F$ is universal for the
function $G$. It is clear that the range of $h$ must uncountable.
Hence there is
$\xi$ such that $h_\xi\subseteq h$. Then
$G(r_\xi,r) = F(h(r_\xi),h_\xi(r))$
for all $r\in A_\xi\cup B_\xi$.

If $f$ is the function defined by $f(y) = F(h(r_\xi),y)$ then $f$
must be Baire 1 and, in particular, defining
$$C = \overline{\SetOf{h_\xi(r)}{r\in A_\xi}} =
\overline{\SetOf{h_\xi(r)}{r\in B_\xi}}
$$
it follows that
$f\restriction C$ is Baire 1. However,
$$f(h_\xi(r)) = F(h(r_\xi),h_\xi(r)) =
G(r_\xi,r) = 1 \mbox { for }r \in A_\xi.$$
Similarly $f(h_\xi(r)) = 0$
for $ r\in B_\xi$.
This is impossible for any Baire class 1 function on the perfect set $C$.
\qed

The techniques of Miller \cite{millergensous} can be
used to produce models with an analytic universal function
but no Borel universal function.

\section{Universal Functions and Martin's Axiom}\label{MA}

Martin's Axiom implies that there are universal
functions on the reals of Baire class 2, see Proposition \ref{proprao}.
Here we show that weakening of Martin's axiom is not strong enough.

\begin{lemma}\label{l:m}
If there are models of set theory $\{{\mathfrak M}_a\}_{a\in \rcantorpar^3}$ 
such that:
\begin{enumerate}
\item $a\in {\mathfrak M}_a$ for each $a\in \rcantorpar^3$
\item $\Reals\not\subseteq {\mathfrak M}_a$ for each $a\in \rcantorpar^3$
\item for any $h:\Reals \to \Reals$ and any $x\in \Reals$ there are reals $y$ and $z$ such that $\{h(y),h(z)\}\subseteq  {\mathfrak M}_{(x,y,z)}$
\end{enumerate}
then there is no Borel universal function.  Moreover, the models ${\mathfrak M}_a$ need only be models of a sufficiently large fragment of set theory to code Borel sets by reals.
\end{lemma}
\begin{proof}
Suppose that $F$ is a Borel universal function. Let $x$ be  a real coding it. Define $G(y,z)$ to be any element of $\Reals \setminus {\mathfrak M}_{(x,y,z)}$. Then if  $h:\Reals\to \Reals$ it is possible to find reals $y$ and $z$ such that $\{h(y),h(z)\}\subseteq  {\mathfrak M}_{(x,y,z)}$. But then, since ${\mathfrak M}_{(x,y,z)}$ is a model of set theory, it follows that $F\in {\mathfrak M}_{(x,y,z)}$ and hence $F(h(y),h(z))\in {\mathfrak M}_{(x,y,z)}$. Since $G(y,z)\notin {\mathfrak M}_{(x,y,z)}$ it follows that $F(h(y),h(z))\neq G(y,z)$ and hence $F$ can not be universal.
\end{proof}

\begin{theor}\label{t:p}
If there is a model of set theory then there is a model of set theory in which there is no
Borel universal function. Indeed, there is no Borel universal function in any model obtained by forcing with a finite support product of $\kappa^+$ ccc partial orders if $\kappa$ has uncountable cofinality.
\end{theor}
\begin{proof}
Let $\Poset_\alpha$ be a ccc partial order for each $\alpha \in \kappa^+$ and suppose
that $$G\subseteq \prod_{\alpha\in \kappa^+}\Poset_\alpha$$
is generic over $V$. Since the finite support iteration adds reals, by taking products of countably many $\Poset_\alpha$ it may as well be assumed that each $\Poset_\alpha$ adds a real.
 For any $\Gamma\subseteq \kappa^+$ let $V_\Gamma$ denote the model $V[G\cap \prod_{\alpha\in \Gamma}\Poset_\alpha]$. For any $x\in \Reals$ in $V[G]$ let $\mu(x)$ be the least ordinal such that $x \in V_{\mu(x)}$.

Given  $(x,y,z)\in \rcantorpar^3$ suppose first that there is no $\theta\in \kappa$ such that $\mu(y) = \mu(x) + \theta$. In this case
 define ${\mathfrak M}_{(x,y,z)} = V_\xi$ where $\xi$ is the largest of $\mu(x)$, $\mu(y)$ and $\mu(z)$. The ccc guarantees that $\xi< \kappa^+$ and the new reals added ensure that  (1) and (2) of Lemma~\ref{l:m} hold. Otherwise, let $\theta(x,y)\in \kappa$ be
such that $\mu(y) = \mu(x) + \theta(x,y)$. Let $\Gamma_{(x,y,z)} = \mu(z) + \theta(x,y) \cup (\kappa^+\setminus (\mu(z) + \kappa))$ and let ${\mathfrak M}_{(x,y,z)} = V_{\Gamma_{(x,y,z)}}$. It is again clear that (1) and (2) of Lemma~\ref{l:m} hold.

To see that (3) holds suppose that $h:\Reals \to \Reals$ and $x\in \Reals$ are in $V[G]$.
For each $\eta\in \kappa$ let $y_\eta\in \Reals$ be such that $\mu(y_\eta) = \mu(x) + \eta$. In other words, $\theta(x,y_\eta) = \eta$. Using the ccc, find $\beta$ so large that
$h(y_\eta)\in V_\beta$ for each $\eta \in \kappa$.  Now let $z\in \Reals$ be such that $\mu(z) = \beta$ and find $\eta\in \kappa$ large enough that $h(z) \in V_\Gamma$
where $\Gamma = \beta + \eta \cup (\kappa^+\setminus (\beta + \kappa))$.
It follows that ${\mathfrak M}_{x,y_\eta,z} = V_\Gamma$ and hence
$\{h(y_\eta),h(z)\}\subseteq {\mathfrak M}_{x,y_\eta,z} $. Hence (3) of Lemma~\ref{l:m} is also satisfied and the result now follows from Lemma~\ref{l:m}.
\end{proof}

\begin{theor}
If there is a model of set theory then there is a model of set
theory in which there is no
Borel universal function yet ${\text MA}_{\aleph_1}$ holds.
\end{theor}

\begin{proof}
Obtain the model of ${\text MA}_{\aleph_1}$ by iterating to $\omega_3$
 with ccc partial orders of size $\aleph_1$ over a model of the Continuum Hypothesis.
  To be precise, let $\{\Poset_\alpha\}_{\alpha\in \omega_3}$ be  names for the ccc partial orders such that $\Rationals_\alpha$ is the iteration of  $\{\Poset_\xi\}_{\xi\in\alpha}$    and $\Poset_\alpha$ is a $\Rationals_\alpha$ name for a ccc partial order of cardinality $\aleph_1$. A set $\Gamma\subseteq \omega_3$ will be called {\em full} if
  for each $\gamma\in \Gamma$ all the conditions in the name $\Poset_\gamma$
  have support contained in $\Gamma\cap \gamma$. If $\Gamma$ is full, let $\Rationals_\Gamma$ be the iteration of only the partial orders $\Poset_\gamma$ for $\gamma\in \Gamma$.

  Cardinal arithmetic and the ccc guarantee that the partial order $\Rationals_{\omega_3}$ has the property that for any subset of $\Rationals_{\omega_3}$ of cardinality $\aleph_1$ is contained in completely embedded partial order of the form $\Rationals_\Gamma$ where $\Gamma$ is a full set of cardinality $\aleph_1$. Even more,
  for any $\xi\in \omega_3$ if $W\subseteq \Rationals_{\omega_3}$ is such that
  $W\setminus \Rationals_\xi$ has cardinality $\aleph_1$ it is possible to find
 a full $\Gamma$ such that $\Gamma \setminus \Rationals_\xi$ has cardinality $\aleph_1$ and $\Rationals_\Gamma$ is completely embedded in $\Rationals_{\omega_3}$.
  Using this, it is possible to mimic the proof of Theorem~\ref{t:p}.

  Let $G\subseteq \Rationals_{\omega_3}$ be generic and for any full $\Gamma\subseteq \omega_3$ such that $\Rationals_\Gamma$ is completely embedded in $\Rationals_{\omega_3}$ let $V_\Gamma = V[G\cap \Rationals_\Gamma]$.
 For any $x\in \Reals$ in $V[G]$ let $\mu(x)$ be the least ordinal such that $x \in V_{\mu(x)}$.

Given  $(x,y,z)\in \rcantorpar^3$ suppose first that there is no $\theta\in \omega_2$ such that $\mu(y) = \mu(x) + \theta$. In this case
 define ${\mathfrak M}_{(x,y,z)} = V_\xi$ where $\xi$ is the largest of $\mu(x)$, $\mu(y)$ and $\mu(z)$.  Otherwise, let $\theta(x,y)\in \omega_2$ be
such that $\mu(y) = \mu(x) + \theta(x,y)$.
There is some $\Gamma_{(x,y,z)}\subseteq \omega_3$ such that
\begin{enumerate}
\item $\mu(z)+ \theta(x,y)\subseteq \Gamma_{(x,y,z)}$
\item $| \Gamma_{(x,y,z)}\setminus \mu(z)| = \aleph_1$
\item  $\Gamma_{(x,y,z)}$ is full
\item $\Rationals_{\Gamma_{(x,y,z)}}$ is completely embedded in $\Rationals_{\omega_3}$.
\end{enumerate}
The let ${\mathcal G}$ be the family of all $\Gamma\subseteq \omega_3$ such that
\begin{enumerate}
\item $\Gamma \cap \mu(z) + \omega_2 = \Gamma_{(x,y,z)}\cap \omega_2$
\item  $\Gamma$ is full
\item $\Rationals_{\Gamma}$ is completely embedded in $\Rationals_{\omega_3}$.
\end{enumerate}
Let  ${\mathfrak M}_{(x,y,z)} = \bigcup_{G\in\mathcal G}V_{\Gamma}$ and note that it is a model of sufficiently much set theory to code Borel sets by reals.
 It is again clear that (1) and (2) of Lemma~\ref{l:m} hold.

To see that (3) holds suppose that $h:\Reals \to \Reals$ and $x\in \Reals$ are in $V[G]$.
For each $\eta\in \omega_2$ let $y_\eta\in \Reals$ be such that $\mu(y_\eta) = \mu(x) + \eta$.  Using the ccc, find $\beta$ so large that
$h(y_\eta)\in V_\beta$ for each $\eta \in \omega_2$.  Now let $z\in \Reals$ be such that $\mu(z) = \beta$ and find $\eta\in \omega_2$ large enough that $h(z) \in V_{\Gamma_{(x,y_\eta,z)}}$.
It follows that ${\mathfrak M}_{x,y_\eta,z} \supseteq V_{\Gamma_{(x,y_\eta,z)}}$ and hence
$\{h(y_\eta),h(z)\}\subseteq {\mathfrak M}_{x,y_\eta,z} $. Hence (3) of Lemma~\ref{l:m} is also satisfied and the result now follows from Lemma~\ref{l:m}.
\end{proof}

\section{Universal Functions of special kinds} \label{special}

Elementary functions in the calculus of two variables can
be obtained from addition $x+y$, the elementary functions of one variable
and closing under composition.  For example, $xy={1\over 2}((x+y)^2-x^2-y^2)$.
We might ask if there could be a universal function of the
form: 
$F(x,y)=k(x+y)$.  By this we mean that for any $G(x,y)$ we can find
$u(x)$ and $v(y)$ such that $G(x,y)=k(u(x)+v(y))$.

\begin{prop}
If there is a universal function, then there is one of the form
$F(x,y)=k(x+y)$,
where $k$ has the same complexity as the given universal function.
\end{prop}
\proof
For simplicity assume that $x+y$ refers to the pointwise addition
in $2^\om$.  A similar argument can be given for ordinary addition
on the real line.

Suppose $F^*:2^\om\times 2^\om\to 2^\om$ is a universal function,
i.e, for every $f:2^\om\times 2^\om\to 2^\om$ there are
$g,h$ with $f(x,y)=F^*(g(x),h(y))$ all $x,y\in 2^\om$.  Given
any $u\in 2^\om$ let $u_0$ be $u$ shifted onto the even
coordinates, i.e, $u_0(2n)=u(n)$ and $u_0(2n+1)=0$.  Similarly
for $v\in 2^\om$ let $v_1$ be $v$ shifted onto the odd coordinates.
Note that $(u,v)$ is easily recovered from $u_0+v_1$.
Hence we can define $k$ by $k(w)=F^*(u,v)$ where $w=u_0+v_1$.
\qed

\begin{prop}
Suppose that there is a universal function $F:2^\om\times 2^\om\to 2^\om$.
Then there exists a function
$f:2^\om\to 2^\om$ such that for every symmetric
$H:2^\om\times 2^\om\to 2^\om$
there exists a $g:2^\om\to 2^\om$ such that $H(x,y)=f(g(x)+g(y))$
for every two distinct $x,y\in 2^\om$.  Furthermore if $F$ is
Borel, then $f$ is Borel.
\end{prop}
\proof
Let $P_s\su\om$ for $s\in 2^{<\om}$ partition $\om$ into infinite
sets.  We say that $y:P_s\to 2$ codes $x:\om\to 2$ iff
$y(a_n)=x(n)$ where $a_0<a_1<a_2<\ldots$ is the increasing
listing of $P_s$.

For any $x\in 2^\om$ define $q(x)\in 2^\om$ so that
$q(x)\res P_{x\res n}$ codes $x$ for
every $n<\om$ and $q(x)\res P_s$ is identically 0 for
any $s$ which is not an initial segment of $x$.

By assumption for any $H:2^\om\times 2^\om\to 2^\om$ there exists
$h$ such that $$H(x,y)=F(h(x),h(y))$$ for all $x,y\in 2^\om$.
Define $g(x)=q(h(x))$.  Without loss of generality we may assume
that $h$ is one-to-one and never identically 0.
Notice for $x\neq y$ that we may easily recover
$h(x)$ and $h(y)$ from $q(x)+q(y)$.  (There will be exactly two
infinite paths in the set of all $s\in 2^{<\om}$ such that
$(q(x)+q(y))\res P_s$ is not identically $0$).   Hence, we
may define $f$ so that
$$f(g(x)+g(y))=F(h(x),h(y)).$$

\qed

\begin{prop}
There does not exist a Borel function $F:2^\om\times 2^\om\to 2^\om$
such that for every Borel $H:2^\om\times 2^\om\to 2^\om$ there
exists Borel $g,h:2^\om\to 2^\om$ with
$$H(x,y)=F(g(x),h(y))$$
for every $x,y\in 2^\om$.
\end{prop}
\proof
Suppose $F$ is a Baire level $\al$ function and let
$U\su 2^\om\times 2^\om$ be a universal ${\bf\Sigma}_{\al+3}^0$
set.  Let $H$ be the characteristic function of $U$.  Given
any Borel $g,h$ let $P\su 2^\om$ be perfect set on which
$h$ is continuous.  Fix $x_0$ so that $U_{x_0}\su P$ is
not ${\bf \Delta}_{\al+3}^0$.  But if we define
$$q:P\to 2^\om \mbox{ by } q(y)=F(g(x_0),h(y))$$
then $q$ is Baire level $\al$ and $U_{x_0}=q^{-1}(1)$
which is a contradiction.
\qed
This proof is similar to Mansfield and Rao's Theorem
\cite{mansfield,mansfield2,rao2} that
the universal analytic set in the plane is not in the $\si$-algebra
generated by rectangles with measurable sides.  See also,
Miller \cite{millerrectang}.

\begin{ques}
Does there always exists a Borel function $$F:2^\om\times 2^\om\to 2^\om$$
such that for every Borel $H:2^\om\times 2^\om\to 2^\om$ there
exists $g,h:2^\om\to 2^\om$ with
$$H(x,y)=F(g(x),h(y))$$
for every $x,y\in 2^\om$?
\end{ques}
Maybe Louveau's Theorem \cite{louv} is relevant for this question.

Stevo Todorcevic has noted the following version of universal
functions:

There exists continuous functions $F_n:2^\om\times 2^\om\to 2^\om$ for
$n<\om$ with the property that for every $G:\pp\times\pp \to 2^\om$ there
exists $h:\pp\to 2^\om$ such that for every  $\al,\be<\pp$
$$G(\al,\be)=\lim_{n\to\infty}F_n(h(\al),h(\be))$$
Where $\pp$ is the least cardinal for which MA $\si$-centered
fails.

We prove that this sort of universal function is equivalent
to a level 2 universal function.

\def\csqr{2^\om\times 2^\om}

\begin{prop}
For any cardinal $\ka$ the following are equivalent:

(1) There exists continuous functions $F_n:\csqr\to 2^\om$ for
$n<\om$ with the property that for every $G:\ka\times\ka \to 2^\om$ there
exists $h:\ka\to 2^\om$ such that
$$G(\al,\be)=\lim_{n\to\infty}F_n(h(\al),h(\be)) \mbox{ for every }
\al,\be \in \ka.$$

(2) There exists a level-2 function $F:\csqr\to 2^\om$
with the property that for every $G:\ka\times\ka \to 2^\om$ there
exists $h:\ka\to 2^\om$ such that
$$G(\al,\be)=F(h(\al),h(\be)) \mbox{ for every }
\al,\be \in \ka.$$

\end{prop}

\proof
Recall that $F$ is level 2 means that for every $n$ the
set $$\{(x,y)\;:\; F(x,y)(n)=1\} \mbox{ is }F_\si.$$

\noindent $(1)\rightarrow (2)$. Given the sequence $F_k$ of continuous functions
define:

$F(x,y)(n)=1$ iff $F_k(x,y)(n)=1$ for all but finitely many $k<\om$.

\bigskip
\noindent $(2)\rightarrow (1)$.
For any $G$ let $G_0$ be $G$ and define $G_1$ by
$G_1(\al,\be)(n)=1-G_0(\al,\be)(n)$.  That is, we switch
$0$ and $1$ on every coordinate of the output.
It follows that we have $h_0$ and $h_1$ such that
for ever $\al,\be<\ka$ and $n<\om$

\par\noindent $G(\al,\be)(n)=1 $ implies
\par $F(h_0(\al),h_0(\be))(n)=1$ and $F(h_1(\al),h_1(\be))(n)=0$
\par\noindent $G(\al,\be)(n)=0 $ implies
\par $F(h_0(\al),h_0(\be))(n)=0$ and $F(h_1(\al),h_1(\be))(n)=1$

For each $n$ define the pair of (nondisjoint) $F_\si$ sets
$P^0_n$ and $P^2_n$ by
$$(\pair(u_0,u_1),\pair(v_0,v_1))\in P^i_n \rmiff F(u_i,v_i)(n)=1$$
By the reduction property for each $n$ there are disjoint $F_\si$ sets
$Q^0_n, Q^1_n$ with
$Q^i_n\su P^i_n$ and $Q^0_n\cup Q^1_n =  P^0_n\cup P^1_n$.
Write each $Q^i_n$ as an increasing sequence of closed sets
$Q^i_n=\bigcup_k C^i_{n.k}$.  Since $C^0_{n.k}$ and $C^1_{n.k}$
are disjoint closed sets, there is a clopen set $D_{n.k}$
with $C^0_{n.k}\su D_{n.k}$ and $C^1_{n.k}$ disjoint from
$D_{n.k}$.

Define the continuous map $F_k$ as follows:
$$F_k(u,v)(n)=1 \rmiff (u,v)\in D_{n,k}$$

Now we verify that this works.  Given $G$ take $h_0$ and $h_1$
as above and define $h(\ga)=\la h_0(\ga),h_1(\ga)\ra$.  Then
for any $\al,\be,n$

If $G(\al,\be)(n)=1$, then $\la h(\al),h(\be)\ra\in P_n^0\sm P_n^1$
so  $\la h(\al),h(\be)\ra\in Q_n^0$ and so
$F_k(\la h(\al),h(\be)\ra)(n)=1$ for all but finitely many
$k$.

If $G(\al,\be)(n)=0$, then $\la h(\al),h(\be)\ra\in P_n^1\sm P_n^0$
so  $\la h(\al),h(\be)\ra\in Q_n^1$ and so
$F_k(\la h(\al),h(\be)\ra)(n)=0$ for all but finitely many
$k$.
\qed

Davies \cite{davies} showed that the continuum hypothesis
implies that the function
$$F(\vec{x},\vec{y})=\sum_{n<\om} x_ny_n$$
has a universal property:
for every $H:\rr\times\rr\to \rr$ there
are functions $f_n,g_n$ for $n<\om$ such that
$$H(x,y)=\sum_{n<\om}f_n(x)g_n(y)$$
for all $x,y\in \rr$.  Moreover the sum has only finitely
many nonzero terms.
Shelah \cite{roslan} remarks that Davies result is false in the
Cohen real model.

\section{Abstract Universal Functions}\label{abstract}

\begin{theorem} \label{existuniv}
If $\ka$ is an infinite cardinal such that $2^{<\ka}=\ka$, then there
is a universal function
$F:\ka \times \ka\to \ka$.
\end{theorem}
\proof
Choose $\rho_\al:\ka\times\ka\to\ka$ for $\al<\ka$ with the
property that for every $\be<\ka$ and $k:\be\to\ka$ there
exists a $\al<\ka$ with $k=\rho_\al\res\be$.
Let $\pair(,):\ka\times\ka\to\ka$ be a bijective pairing function.
Define
$F:\ka\times\ka\to\ka$ as follows:
 $$F(\pair(\al_0,\al_1),\pair(\be_0,\be_1))=\left\{
\begin{array}{ll}
\rho_{\be_1}(\al_0) & \mbox{ if } \al_0\leq\be_0 \\
\rho_{\al_1}(\be_0) & \mbox{ if } \al_0>\be_0 \\
\end{array}\right.$$
To see that $F$ is universal, let $H:\ka\times\ka\to\ka$
be arbitrary.  For each $\be$ choose $h(\be)$ so
that $H(\al,\be)=\rho_{h(\be)}(\al)$ for all $\al\leq\be$.
Similarly, choose $g(\al)$ so that
$H(\al,\be)=\rho_{g(\al)}(\be)$ for all $\be<\al$.
If follows that n
$$H(\al,\be)=F(\pair(\al,g(\al)),\pair(\be,h(\be)))$$ for
all $\al,\be<\ka$.
\qed

\begin{remark}\label{remarkomega}
For example, there is a universal $F:\om\times\om\to\om$.
\end{remark}

\begin{remark}
Theorem \ref{existuniv} is probably just a special case of
Theorem 6 of Rado \cite{rado}.
\end{remark}

\begin{define}
Let $\fin(X)$ be the partial order of  partial functions
from a finite subset of $X$ into 2.  Let $\ctbl(Y)$ be the
partial order\footnote{These posets are denoted
in Kunen\cite{kunen} by $Fn(X,2)$ and $Fn(Y,2,\om_1)$.}
of countable partial functions
from $Y$ into 2.
\end{define}

\begin{theorem} It is relatively consistent with ZFC that
there is no universal function $F:\cc \times \cc\to \cc$.
\end{theorem}
\proof
In our model $\cc=\om_2$ and there is no $F:\om_2\times\om_2\to 2$
with the property that for every $f:\om_2\times\om_1\to 2$
there exists $g_1:\om_2\to\om_2$ and $g_2:\om_1\to\om_2$
such that $f(\al,\be)=F(g_1(\al),g_2(\be))$ for every $\al<\om_2$
and $\be<\om_1$.

Let $M$ be a countable transitive model of ZFC + GCH.
Force with $\ctbl(\om_3)$ followed by $\fin(\om_2)$.
Let $G$ be $\ctbl(\om_3)$-generic over $M$ and
$H$ be $\fin(\om_2)$-generic over $M[G]$.  We will
show there is no $F$ in the model $N=M[G][H]$.

By standard arguments\footnote{Kunen\cite{kunen} p.253
Solovay \cite{solovay} p.10}
involving iteration and product forcing we may regard
$N$ as being obtained by forcing with $\ctbl(\om_3)^M$ over
the ground model $M[H]$.  Of course, in $M[H]$ the
poset $\ctbl(\om_3)^M$ is not countably closed but it
still must have the $\om_2$-cc.  Hence for
any $F:\om_2\times\om_2\to 2$ in $N$ we may find
$\ga<\om_3$ such that $F\in M[H][G\res\ga]$.

Use $G$ above $\ga$ to define $f:\om_2\times\om_1\to 2$, i.e.,
$$f(\al,\be)=G(\ga+\om_1\cdot\al+\be).$$
Now suppose for contradiction that in $N$ there were
$g_1:\om_2\to\om_2$ and $g_2:\om_1\to\om_2$
such that $f(\al,\be)=F(g_1(\al),g_2(\be))$ for every $\al<\om_2$
and $\be<\om_1$.
Using the $\om_2$ chain condition there would be $I\su\om_3$
in $M$ of size $\om_1$ such that $g_2\in M[H][G\res(\ga\cup I)]$.
Choose $\al_0<\om_2$ so that
$\ga\cup I$ is disjoint from $$D=\{\ga+\om_1\cdot\al_0+\be\st \be<\om_1\}.$$

It easy to see by a density argument that the function
$G\res D$ is not in $M[H][G\res(\ga\cup I)]$.  But
this is a contradiction, since $G\res D$ is easily defined
from the function $f(\al_0,\cdot)$,
$f(\al_0,\be)=F(g_1(\al_0),g_2(\be))$ for all $\be$,
and $F,g_2$ are in $M[H][G\res(\ga\cup I)]$.
\qed

\begin{ques}
Is it consistent with $2^{<\cc}>\cc$ to have a universal
function $F:2^\om \times 2^\om\to 2^\om$?  How about a Borel $F$?
\end{ques}

\begin{theorem}
Suppose MA$_{\om_1}$.  Then there exists $F:\om_1\times\om\to\om_1$ which
is universal, i.e., for every $f:\om_1\times\om\to\om_1$ there exists
$g:\om_1\to\om_1$ and $h:\om\to\om$ such that
$$f(\al,n)=F(g(\al),h(n)))\mbox{ for every } \al<\om_1\rmand n<\om .$$
\end{theorem}

\proof
There is an obvious notion of universal $F:\al\times\be\to\gamma$.
We produce a universal $F:\om_1\times\om\to\om$ and then show that
this is equivalent to the existence of a universal $F:\om_1\times\om\to\om_1$.

Standard arguments, show that there exists a family $h_\al:\om\to\om$
for $\al<\om_1$ of independent functions, i.e., for
any $n$, $\al_1<\al_2<\cdots<\al_n<\om_1$ and $s:\{1,\ldots,n\}\to\om$ there
are infinitely many $k<\om$ such that
$$h_{\al_1}(k)=s(1)$$
$$h_{\al_2}(k)=s(2)$$
$$\vdots$$
$$h_{\al_n}(k)=s(n).$$
Define $H:\om_1\times\om\to\om$ by $H(\al,n)=h_\al(n)$.
We show that $H$ is universal mod finite, in sense which will
be made clear.
Given any $f:\om_1\times\om\to\om$ define the following
poset $\poset$.  A condition $p=(s,F)$ is a pair
such that $s\in\om^{\om}$ is one-to-one
and $F\in [\om_1]^{<\om}$.
We define $p\leq q$ iff
\begin{enumerate}
\item $s_q\su s_p$,
\item $F_q\su F_p$, and
\item $f(\al,n)=h_\al(s_p(n))$
for every $\al\in F_q$ and $n\in\dom(s_p)\sm\dom(s_q)$.
\end{enumerate}
It is easy to see that $\poset$ is ccc, in fact, $\si$-centered
since any two conditions with the same $s$ are compatible.
Since the family $(h_\al:\om\to\om:\al<\om_1)$ is independent,
for any $p\in\poset$ there are extensions of $p$ with arbitrarily
long $s$ part.   It follows from MA$_{\om_1}$ that there exists
$h:\om\to\om$ with the property that for every $\al<\om_1$
for all but finitely many $n$ that $f(\al,n)=h_\al(h(n))$.

To get a universal map $F:\om_1\times\om\to\om$, simply take
any $F$ with the property that for every $\al<\om_1$
and any $h^\pr=^*h_\al$ (equal mod finite)
there is $\be$ such that $F(\be,n)=h^\pr(n)$ for every $n$.
Since the function $h$ is one-to-one, it easy to find
$k:\om_1\to\om_1$ such that $F(k(\al),h(n))=f(\al,n)$ for
all $\al$ and $n$.

Finally we show that having a universal $F:\om_1\times\om\to\om$
gives a universal $F^\pr:\om_1\times\om\to\om_1$.
For any infinite $\al<\om_1$ fix a bijection from $j_\al:\om\to\al$.
Construct $F^\pr$ with the property that for every pair
$\al,\be<\om_1$ there are uncountably many $\ga<\om_1$ such
that $F_\ga^\pr=j_\be\circ F_\al$, i.e.,
$$F^\pr(\ga,n)=j_\be(F(\al,n)) \mbox{ for all } n<\om.$$
Now we verify that $F^\pr$ is universal.  Let
$f^\pr:\om_1\times\om\to\om_1$ be arbitrary. Let
$$i_\al=\om+\sup \{f^\pr(\al,n)+1\st n<\om\}.$$
Define $f$ into $\om$ by $f(\al,n)=j_{i_\al}^{-1}(f^\pr(\al,n))$.
Since $F$ is universal there exists $g,h$ with
$$F(g(\al),h(n))=f(\al,n)=j_{i_\al}^{-1}(f^\pr(\al,n)).$$
By our definition of $F^\pr$ we may construct $g^\pr$ so
that
$$F^\pr(g^\pr(\al),h(n))=j_{i_\al}(F(g(\al),h(n)))$$
and we are done since
$$j_{i_\al}(F(g(\al),h(n)))=f^\pr(\al,n).$$
\qed

\begin{ques}
Does MA$_{\om_1}$ imply there exists $F:\om_1\times\om_1\to\om_1$ which
is universal?  Is it consistent one way or the other?
This question may be related to Shelah results on universal
graphs of size $\om_1$, see Shelah \cite{shelah,shelahsmall,shelahcorrection}.
\end{ques}

\begin{prop}
If there is a universal function $F:\ka\times\ga\to 2$ then
for every $n<\om$ there is a universal function
$F:\ka\times\ga\to n$.
\end{prop}
\proof
We produce a $F^*$ which is universal for $n=2\times 2$.
For any $H_1,H_2:\ka\times\ga\to 2$
there exists $g_1,h_1,g_2,h_2$ such that
$$H_1(\al,\be)=F(g_1(\al),h_1(\be)) \rmand
H_2(\al,\be)=F(g_2(\al),h_2(\be))$$
for all $\al\in\ka$ and $\be\in\ga$.  Now define
$$F^*(\pair(\al_1,\al_2),\pair(\be_1,\be_2))=
\la F(\al_1,\be_1),F(\al_2,\be_2)\ra$$
and
$$g(\al)=\pair(g_1(\al),g_2(\al)) \rmand
h(\be)=\pair(h_1(\be),h_2(\be)).$$
Note that
$$F^*(g(\al),h(\be))=\pair(H_1(\al,\be),H_2(\al,\be)$$
for all $\al\in\ka$ and $\be\in\ga$.
\qed

\section{Higher Dimensional Universal Functions}\label{higher}

\def\cantor{{2^\om}}
\begin{define}
A $k$-dimensional universal function is a function $$F:(\cantor)^k \to \cantor$$
such that for every function $G:(\cantor)^k \to \cantor$ there is
$h:\cantor \to \cantor$ such that
$$G(x_1,x_2,\ldots,x_k) = F(h(x_1),h(x_2),\ldots,h(x_k))$$
for all $(x_1,x_2,\ldots,x_k) \in (\cantor)^k$.
\end{define}

\begin{prop}
Suppose $F(x,y)$ is a universal function, then $F(F(x,y),z)$ is
a 3-dimensional universal function.
\end{prop}
\proof
Given $G(x,y,z)$ define $G_0(u,z)=G(u_0,u_1,z)$ using unpairing,
$u=\pair(u_0,u_1)$.  By universality of $F$ there are
$g,h$ with $G_0(u,z)=F(g(u),h(z))$.  Again by universality
of $F$ there are $g_0,g_1$ with $g(\pair(u_0,u_1))=F(g_0(u_0),g_1(u_1))$
and hence $G(x,y,z)=F(F(g_0(x),g_1(y)),h(z))$.
\qed
Hence the existence of a universal function in dimension 2 is
equivalent to the existence of a universal function in dimension k
for any $k>1$.   Note however that the Baire complexity of
$F(F(x,y),z)$ is higher than that of $F$.

We may also consider universal functions $F$ where the parameters
functions are functions of more than one variable, for
example: $$\forall G\;\exists g,h,k\;\forall x,y,z\;\;\;
G(x,y,z)=F(g(x,y),h(y,z),k(z,x)).$$   Although this easily
follows from the existence of a dimension 3 universal, we
do not know if it is equivalent.
The reader will easily be able to imagine many variants.
For example,

$G(x,y,z)=F(g(x,y),h(y,z))$

$G(x_1,x_2,x_3,x_4)=F(g_1(x_1,x_2),g_2(x_2,x_3),g_3(x_3,x_4),g_4(x_4,x_1))$

\noindent where we have omitted quantifiers for clarity.
These two variants are equivalent to the existence of
2-dimensional universal function.  To see this in the
first example put $y=0$ and
get

$G(x,z)=F(g(x,0),h(0,z))$.

\noindent In the second example put $x_2=x_4=0$ and get

$G(x_1,x_3)=F(g_1(x_1,0),g_2(0,x_3),g_3(x_3,0),g_4(0,x_1))$.

\noindent
More generally, suppose $F$ and $\vec{x_k}$'s have the property that
for every $G$ there are $g_k$'s such that for all $\vec{x}$
$$G(\vec{x})=F(g_1(\vec{x}_1),\ldots,g_n(\vec{x}_n)).$$
Suppose there are two variables $x$ and $y$ from $\vec{x}$ which
do not simultaneously belong to any $\vec{x}_k$.  Then we
get a universal 2-dimensional function simply by putting
all of the other variables equal to zero.

\begin{prop}
If there  is a $(3,2)$-dimensional universal function, i.e., an $F(x,y,z)$
such that for every $G$ there is $h$ with
$$G(x,y,z)=F(h(x,y),h(y,z),h(z,x))\mbox{ all }x,y,z$$
then for every $n>3$ there is a $(n,2)$-dimensional universal
function $F$, i.e., for every $G$ n-ary there is a binary $h$ with
$$G(x_1,x_2,\ldots,x_n)=F(\la h(x_i,x_j):1\leq i<j\leq n\ra)
\mbox{ all }\vec{x}.$$
$F$ is $\comb(n,2)$-ary.
Conversely, if there is a $(n,2)$-dimensional universal
function for some $n>3$, then there is a $(3,2)$-dimensional universal
function.
\end{prop}
\proof
Consider the case for $n=4$.

Suppose that $F$ is  $(3,2)$-dimensional
universal function.  Given a 4-ary function $G(x,y,z,w)$ for each
fixed $w$ we get a function $h_w(u,v)$ with
$$G(x,y,z,w)=F(h_w(x,y),h_w(y,z),h_w(z,x))\mbox{ for all } x,y,z.$$
But now considering $h(u,v,w)=h_w(u,v)$ we get a function
$k(s,t)$ with $h(u,v,w)=F(k(u,v),k(v,w),k(w,u))$.
Note that

\noindent $\;\;\;\;\;\;G(x,y,z,w)=$

$\;\;\;\;\;\;F(F(k(x,y),k(y,w),k(w,x)),$

$\;\;\;\;\;\;\;\;\;\;\;\;F(k(y,z),k(z,w),k(w,y)),$

$\;\;\;\;\;\;\;\;\;\;\;\;\;\;\;\;\;\;F(k(z,x),k(x,w),k(w,z))).$

\noindent Note that
$k(s,t)$ and $k(t,s)$ can
be combined by pairing and unpairing into a single function $k_1(s,t)$.
From this one can define a $(4,2)$-dimensional
universal function.

For the converse, if $F$ is a $(4,2)$-dimensional
universal function, then for every $G$ 3-ary, there exists
$h$ binary with
$$G(x,y,z)=F(h(x,y),h(y,z),h(x,z),h(x,0),h(y,0),h(z,0)).$$
But note that, for example, $h(x,y)$ and $h(x,0)$ can
be combined into a single function of $h_1(x,y)$.  Hence
we can get a $(3,2)$-dimensional
universal function.
\qed

Next we state a generalization of these ideas:

\begin{define}\label{def29}
Suppose $\Sigma\su \pow(\{0,1,2,\ldots,n-1\})=\pow(n)$ (the power set of $n$).
Define
$U(\ka,n,\Sigma)$ to mean that there exists
$F:\ka^{\Sigma}\to\ka$ such that
for every $G:\ka^n\to\ka$ there are $h_Q:\ka^{|Q|}\to\ka$
for $Q\in\Sigma$
such that
$$G(x_0,x_1,\ldots,x_{n-1})=F\left(h_Q(x_j:j\in Q):Q\in\Sigma\right)
\mbox{ for all } \vec{x}\in \ka^n.
$$
\end{define}

Then the last two propositions can be generalized to show:

\begin{prop}\label{propmn}
For any infinite cardinal $\ka$ and positive integer $n$
\begin{enumerate}
\item $U(\ka,n+1,[n+1]^n)$ implies $\forall m>n\;\; U(\ka,m,[m]^n)$.
\item ($\exists m>n\;\; U(\ka,m,[m]^n)$) implies $U(\ka,n+1,[n+1]^n)$.
\item $U(\ka,n+1,[n+1]^n)$ implies $U(\ka,n+2,[n+2]^{n+1})$
\end{enumerate}
\end{prop}

\def\minf{\,\cdot\,}

\begin{define}\label{defU}
Define $U(\ka,n)$ to be any of the equivalent
$U(\ka,m,[m]^n)$ for $m>n$.  Note that $n$ is the arity of
the inside parameter functions, the arity of the universal function
is less important.
\end{define}

We will show that $U(\ka, n)$ are the only generalized
multi-dimensional universal functions properties.
Clause (3) says that $U(\ka, n)$ implies $U(\ka, {n+1})$
and we will show that none of these
implications can be reversed.

\begin{prop}\label{sig}
Let $\ka$ be an infinite cardinal, $n\geq 2$, and $\Sigma,\Sigma_0,\Sigma_1$
subsets of $\pow(n)$.
\begin{enumerate}
\item If $\Sigma_0\su \Sigma_1$ , then
$U(\ka,n,\Sigma_0)$ implies $U(\ka,n,\Sigma_1)$.
\item If $Q_0\su Q_1\in\Sigma$, then
$U(\ka,n,\Sigma)$ is equivalent to $U(\ka,n,\Sigma\cup\{Q_0\})$.
\item Suppose $\Sigma$ is closed under taking subsets,
every element of $n$ is in some element of $\Sigma$, and
$n=\{0,1,2,\ldots,n-1\}\notin\Sigma$.  Let $m+1$ be the size
of the smallest subset of $n$ not in $\Sigma$.
Then $U(\ka,n,\Sigma)$ is equivalent to $U(\ka,m)$.
\end{enumerate}
\end{prop}
\proof
$\;\;\;\;$
(1) This is true because the $F$ which works for $\Sigma_0$ also works
for $\Sigma_1$ by ignoring the values of $h_Q$ for $Q\in \Sigma_1\sm\Sigma_0$.

(2) This is true because given $h_{Q_0},h_{Q_1}$ we may define
a new $\hat{h}_{Q_1}$ by outputting the pairing
$$\hat{h}_{Q_1}( x_j:j\in Q_1)=
\la (h_{Q_0}(x_j:j\in Q_0),(h_{Q_1}(x_j:j\in Q_1)\ra$$

(3)
Choose $R\su\{0,1,\dots n-1\}$ not in $\Sigma$ with $|R|=m+1$.
By choice of $m$ all subsets of $R$ of size $m$ are in $\Sigma$.
By setting $x_i=0$ for $i\notin R$, we see that
 $U(\ka, m)$ is true.

Now assume  $U(\ka, m)$ is true.
By Proposition \ref{propmn} we have that $U(\ka,n,\Sigma_0)$ is true
where $\Sigma_0=[n]^m$.
But $\Sigma_0\su \Sigma$ and so by part (1),  $U(\ka,n,\Sigma)$
is true.
\qed

\begin{remark}
For any $n$ and $\Sigma\su \pow(n)$ if 
$\bigcup\Sigma\neq n=\{0,1,2,\ldots,n-1\}$
then  $U(\ka,n,\Sigma)$ is trivially false.  If
$n\in\Sigma$, then $U(\ka,n,\Sigma)$ is trivially true.
If neither of these is true, then by the Proposition \ref{sig}
there exists $m$ with
$U(\ka,n,\Sigma)$ equivalent to $U(\ka, m)$.
\end{remark}

\begin{prop} \label{prop33}
The following are true in ZFC.
\begin{enumerate}
\item $U(\om, 1)$
\item $U(\om_1, 2)$
\item $U(\ka, 1)$ implies $U(\ka^+, 2)$
\item $U(\ka, n)$ implies $U(\ka^+, {n+1})$
\item $U(\om_n, {n+1})$ every $n\geq 0$.
\end{enumerate}
\end{prop}
\proof
For (1) see Remark \ref{remarkomega}.
We prove (2) and leave 3-5 to the reader.

Suppose that $f:\om^2\to\om$ witnesses
$U(\om,2,1)$.  For any countable ordinal $\de>0$ let
$\de=\{\de_i:i<\om\}$. Define
$$F_0(\de,n,m)=\de_{f(n,m)}.$$
Now suppose $G:\om_1^3\to\om_1$.  Define
$$k(\de)=\sup\{G(\al,\be,\ga)\;:\;\al,\be,\ga\leq\de\}+1$$
For any $\ga<\om_1$ let $\ga^*=k(\ga)$.
Define $g:\om^2\to\om$ by
$$G((\ga+1)_n,(\ga+1)_m,\ga))=\ga^*_{g(n,m)}.$$
By the universality property of $f$ there exists
$h:\om\to\om$ with
$$g(n,m)=f(h(n),h(m)) \mbox{ for every } n,m\in\om.$$
For $\de\leq \ga$
define
$h_1(\de,\ga)=h(k)$ where $\de=(\ga+1)_k$.   Then we have
that
$$\forall \al,\be\leq\ga<\om_1\;\;\;
G(\al,\be,\ga)=F_0(k(\ga),h_1(\al,\ga),h_1(\be,\ga)).$$

Define $F$ as follows:

\medskip
\noindent
$F(\al,\be,\ga,\al^*,\be^*,\ga^*,n_1,m_1,n_2,m_2,n_3,m_3)=$
$$
\left\{
\begin{array}{ll}
F_0(\ga^*,n_1,m_1) & \rmif \al,\be\leq\ga \\
F_0(\be^*,n_2,m_2) & \rmif \ga < \be \rmand \al\leq\be \\
F_0(\al^*,n_3,m_3) & \rmif \be,\ga < \al \\
\end{array}\right.
$$

Then given $G$ we can find $k,h_1,h_2,h_3$ so that

\noindent $G(\al,\be,\ga))=$

$F(\al,\be,\ga,k(\al),k(\be),k(\ga),$

$\;\;\;\;\;\;\;\;\; h_1(\al,\ga),h_1(\be,\ga),\;\;
h_2(\al,\be),h_2(\ga,\be),\;\; h_3(\be,\al),h_3(\ga,\al)).
$

\qed

The $\ka$-Cohen real model is any model of ZFC obtained
by forcing with the poset of finite partial functions from $\ka$ to
2 over a countable transitive ground model satisfying ZFC.

\begin{prop}\label{prop34}
In the $\om_2$-Cohen real model
we have that
$U(\om_1,1)$ fails.  Similarly, $U(\om_2,2)$
fails in the $\om_3$-Cohen real model.
More generally, we have that
$U(\ga,n)$ fails in the $\ka$-Cohen real model
when $\ka>\ga\geq \om_n$.
\end{prop}
\proof
We show that $U(\om_2,2)$
fails in the $\om_3$-Cohen real model, leaving the rest
to the reader.

Let $M$ be a countable transitive model of ZFC and
in $M$ define $\poset$ to be the poset of finite partial
maps from $\om_3\times\om_3\times\om_3$ into 2.  We claim that if
$G$ is $\poset$-generic over $M$, then there is
no map $F:\om_2\times\om_2\times\om_2\to \om_2$
which is (3,2)-universal for maps
of the form $H:\om\times\om_1\times\om_2\to 2$.

Suppose for contradiction that $F$ is such a map.
By the ccc we may find $\ga_0<\om_3$ with $F\in M[G\res\ga_0^3]$.
Hence we may find maps $h_1:\om\times\om_1\to\om_3$,
$h_2:\om\times\om_2\to\om_3$, and
$h_3:\om_1\times\om_2\to\om_3$ such that
$$H(n,\be,\ga)=^{def}G(n,\be,\ga_0+\ga)=F(h_1(n,\be),h_2(n,\ga),h_3(\be,\ga)).$$
for every $n<\om,\be<\om_1,\ga<\om_2$.
 By ccc we can choose
$\ga_1<\om_2$ such that $h_1\in M[G^*]$ where
$G^*$ is $G$ restricted $\{(\al,\be,\rho)\in\om^3\;:\;\rho\neq \ga_0+\ga_1\}$.
Define $g:\om\times\om_1\to 2$ by
$$g(n,\al)=G(n,\al,\ga_0+\ga_1)$$
Note that we have that $F,h_1\in M[G^*]$, $g$ is
Cohen generic over $M[G^*]$, and
$$g(n,\al)=F(h_1(n,\al),h_2(n,\ga_0+\ga_1),h_3(\al,\ga_0+\ga_1)).$$

Since the extension by $g$ is ccc, we may find $\al_0<\om_1$
such that
$$h_2\in M[G^*][g\res (\om\times\al_0)]=^{def}N.$$
But this is
a contradiction because $g_{\al_0}$ defined by $g_{\al_0}(n)=g(n,\al_0)$
is Cohen generic over $N$.
But $F,h_1,h_2\in N$ and for any $\ga_2<\om_2$ the map $k$ defined by
$$k(n)=F(h_1(n,\al_0),h_2(n,\ga_0+\ga_1),\ga_2)
\mbox{ for all } n<\om$$
is in $N$ and so can never be equal to $g_{\al_0}$.
Thus $h_3(\al_0,\ga_0+\ga_1)=\ga_2$ cannot be defined.
\qed

\begin{cor}
Let  $\aleph_\om\leq\ga<\ka$.  In the
$\ka$-Cohen real model we have that
$$U(\om_n,n+1)+\neg U(\om_n,n) \mbox{ for all }n>0,$$
and $$\neg U(\ga,n) \mbox{ for all } n>0 .$$
\end{cor}

\begin{remark}
If we start with a model $M_1$ of GCH and force with the countable
partial functions from $\ka=\aleph_{\om+1}$ into $2$ then in
the resulting model $M_2$, we have CH and so
$U(\om_1,1)$ (Theorem \ref{existuniv}).
We get $U(\om_n,n)$  by Propositions \ref{prop33}.  By
an argument similar to Proposition \ref{prop34}
but raised up one cardinal, we have $\neg U(\om_n,n-1)$
for $n\geq 2$.  If we then add $\ka=\om_3$ Cohen reals to $M_2$ to
get $M_3$, then we will have in $M_3$ that $|2^\om|=\om_3$
and $\neg U(\om_3,2)$ by argument of Proposition \ref{prop34}
lifted by one cardinal.   $U(\om_3,4)$ is true in ZFC
by Proposition \ref{prop33}.
This leaves the obvious gap question.
\end{remark}

\begin{define}
In the case of Borel universal functions of higher dimensions,
we use $U({\rm Borel},n)$ to mean
the analogous thing as in Definition \ref{defU} only we require
that the universal map $F$ be Borel.
\end{define}

\begin{prop}
The following are true:
\begin{enumerate}
\item $U(Borel, n)$ implies $U(Borel, n+1)$
\item $U(Borel, \Sigma, n)$ is equivalent to
$U(Borel, m)$ for $m+1$ the size of the smallest
subset of $n$ not in the downward closure of $\Sigma$.
\end{enumerate}
\end{prop}
\proof
The composition of Borel functions is Borel,
and pairing and unpairing functions are continuous.
\qed

We can further refine $U({\rm Borel},n)$
in the special case that our universal
function $F$ is a level $\al$ Borel function.
Since the composition of level $\al$-functions is
not necessarily level $\al$, i.e., $F(F(x,y),z)$ need  be at
the level $\al$ just because $F$ is.  Hence it is not immediately
obvious that the binary case of the next proposition implies
the n-ary case.  The proof here is similar to that
of Rao \cite{rao}.

\begin{prop}\label{proprao}
Assume Martin's Axiom.  Then for every
$n>1$ there is a level 2 Borel function $F:(2^\om)^n\to 2^\om$
which is universal, i.e., for every $G:(2^\om)^n\to 2^\om$
there exists $h_i:2^\om\to 2^\om$ such that for every $x$ in $(2^\om)^n$
$$G(x_1,\ldots,x_n)=F(h_1(x_1),\ldots,h_n(x_n))$$
\end{prop}
\proof

For simplicity we prove it for $n=3$. Let
$$D_1=\{(\al,\be,\ga)\st \al,\be\leq\ga<\cc\}$$
Let $F\su 2^\om\times 2^\om\times 2^\om$ be
an $F_\si$ set with the property that for every $F_\si$ set
$H\su  2^\om\times 2^\om$ there exists $z\in 2^\om$ with
$$H=F_z=^{def}\{(x,y)\st (x,y,z)\in F\}.$$

Let $g:\cc\to 2^\om$
be a 1-1 map.
Recall that Martin's Axiom implies that every set $X\su 2^\om$ with
$X|<\cc$ is a $Q$-set, i.e., every $Y\su X$ is a relative
$F_\si$.  This is due to Silver and can be found in any
standard treatment of Martin's Axiom.
Thus given any $A\su D_1$ we can find $h_1:\cc\to 2^\om$
with
the property that for every $\al,\be\leq\ga$
$$(\al,\be,\ga)\in A\rmiff (g(\al),g(\be),h_1(\ga))\in F$$

Similarly let
$$D_2=\{(\al,\be,\ga)\;:\; \al,\ga\leq \be<\cc\}
\rmand
D_3=\{(\al,\be,\ga)\;:\; \ga,\be\leq \al<\cc\}$$
Now given any $A\su D_2$ or $A\su D_3$.
and obtain $h_2$, and $h_3$ with the
analogous property.

Note that we may determine which case ($D_i$)
in an $F_\si$ way as follows.  Let $k:\cc\to \om^\om$ be
a scale, i.e., if $\al<\be$ then $k(\al)(n)<k(\be)(n)$
for all but finitely many $n<\om$.   Such an object exists
assuming Martin's axiom.   We claim now that there exists
an $F_\si$ predicate $H$ with the property that for
all $A\su \cc\times\cc\times\cc$ there exists
$h:\cc\to 2^\om$ such that
$$\forall \al,\be,\ga\;\; (\al,\be,\ga)\in A \rmiff
(h(\al),h(\be),h(\ga))\in H.$$
To see how to do this note that the function $k$ can tell
us which case we are in $D_1$, $D_2$, or $D_3$.  Then the
function $h$ codes up $k$ and the $h_1,h_2,h_3$.

Similarly, to the proof of Theorem \ref{mainthm}, we get
and $F:2^\om\times 2^\om \times 2^\om\to 2^\om$ which
is a level 2 Borel map which is universal.

\qed

\bigskip

\addresspaul

\addressarn

\addressjuris

\addressbill

\end{document}